\documentclass[12pt, oneside, reqno]{amsart}

\usepackage{amsfonts, amsmath, amssymb, amsthm, amscd}
\usepackage{anysize}
\usepackage{mathtools}

\usepackage[T2A]{fontenc}
\usepackage[utf8]{inputenc}
\usepackage[english]{babel}
\usepackage{cmap}
\usepackage{enumerate}

\usepackage{xcolor}
\usepackage{hyperref}
\hypersetup{
	colorlinks   = true, %Colours links instead of ugly boxes
	urlcolor     = blue, %Colour for external hyperlinks
	linkcolor    = blue, %Colour of internal links
	citecolor    = red %Colour of citations
}

\makeatletter
\@namedef{subjclassname@2020}{%
  \textup{2020} Mathematics Subject Classification}
\makeatother

\theoremstyle{theorem}
\newtheorem{thm}{Theorem}
\newtheorem{prp}{Proposition}[section]
\newtheorem{lem}{Lemma}[section]
\newtheorem{cor}{Corollary}[section]

\theoremstyle{definition}
\newtheorem{dfn}{Definition}[section]
\newtheorem{rem}[dfn]{Remark}

\numberwithin{equation}{section}
\newcommand{\tpdfs}{\texorpdfstring{$s$}{s}}%

\DeclareMathOperator{\dom}{dom}%

\providecommand{\parenth}[1]{\left(#1\right)}%

\def\N{{\mathbb N}}
\def\R{{\mathbb R}}

\def\co{\mathop{\rm co}}

\def\supp{\mathop{\rm supp}}%
\def\closure{\mathop{\rm cl}}%
\def\phi{\varphi}%
\def\epsilon{\varepsilon}%
\newcommand{\st}{:\;}
\newcommand{\norm}[1]{\left\|#1\right\|}
\newcommand{\enorm}[1]{\left|#1\right|}
\newcommand{\volbs}[1][s]{\sthing[#1]{\kappa}_{d+1}}%
\newcommand{\la}{\lambda}

\newcommand{\length}[1]{\mathrm{length}\left(#1\right)}

\providecommand{\abs}[1]{\lvert#1\rvert}%

\def\legendre{\mathcal{L}}%
\def\polar{\circ}
\newcommand{\polarset}[1]{{#1}^{\polar}}%

\newcommand{\slogleg}[2][s]{
\legendre_{ #1} #2 }%

\newcommand{\truncf}[1]{
\left(#1 \right)_{\! \! +}
}%

\newcommand{\gammaf}[1]{\operatorname{\Gamma} \!\left(#1 \right)}%

\newcommand{\shift}[2]{\operatorname{Shift}\left[#1, #2\right]}%
\newcommand{\ball}[1]{\mathbf{B}^{#1}}
\newcommand{\vol}[1]{\operatorname{vol}\nolimits_{#1}}%
\newcommand{\spherem}[1]
{\sigma\left( #1 \right)}%
\newcommand{\di}{\,\mathrm{d}}%integration d
\newcommand{\suppfunc}[2]{h_{#1}\!\parenth{#2}}%

\newcommand{\santalosreg}[3]{\operatorname{S}_{reg}\!\parenth{#1, #2, #3}}
\newcommand{\iprod}[2]{\left\langle#1,#2\right\rangle}%

\newcommand{\sthing}[2][s]{\prescript{(#1)}{}{#2}}%
\newcommand{\slift}[2][s]{\sthing[#1]{\operatorname{Lift}{#2}}}%
\newcommand{\smeasure}[2][s]{\sthing[#1]{\mu}\!\left(#2\right)}%

\usepackage{color}

\newcommand{\grishasays}[1]{}{}%{\textcolor{cyan}{\marrow \textsc{Grisha:} \textsf{#1}}}   
\newcommand{\conv}{\mathrm{conv}}%

\title{Geometric representation of classes of concave functions and duality}
\author{Grigory Ivanov %\address{Grigory Ivanov: 
\and
Elisabeth  M. Werner*} \thanks{* partially supported by NSF Grant DMS-2103482}

\subjclass[2020]{Primary 52A23; Secondary 52A40, 46T12}
\keywords{logarithmically concave function, $s$-concave function, Santal\'o inequality}
\begin{document}
\begin{abstract}
Using a natural representation of a $1/s$-concave function on $\R^d$ as a convex set in $\R^{d+1},$ we derive a simple formula for the integral of its $s$-polar.
This leads to convexity properties of the integral of the $s$-polar function with respect to the center of polarity. In particular, we prove that the reciprocal of the  integral of the polar function of a log-concave function is log-concave as a function of the center of polarity.
Also, we define the Santal\'o regions for $s$-concave and log-concave functions and  generalize the Santal\'o inequality for them in the case the origin is not the Santal\'o point.
\end{abstract}

\maketitle
\section{Introduction}
Log-concave and $s$-concave functions provide  a natural  extension of the theory of convex bodies. 
Starting with 
functional version of the famous Blaschke--Santal\'o inequality  
\cite{artstein2004santalo, artstein2007characterization, KBallthesis, 
fradelizi2007some, lehec2009partitions}, 
much research has been devoted to the study of such functions in recent years.  This  has led to  e.g., 
functional analogs of the floating body \cite{LiSchuettWerner},  John ellipsoids \cite{Alonso-Gutierrez2017, ivanov2020functional} and L\"{o}wner ellipsoids \cite{Alonso-Gutierrez2020, LiSchuettWerner2019}.
 More examples can be found in e.g.,  \cite{ColesantiLudwigMussnig2017, ColesantiLudwigMussnig, Gardner, 
Rotem2020}).
\par
Motivated by  the setting of convex bodies, we investigate  in this paper the properties of log-concave and $1/s$-concave functions  related to  duality transforms associated with the corresponding class of functions.  
\vskip 2mm
 \subsection{Background}
 For $s > 0,$ a non-negative function $f$ on $\R^d$ is called $1/s$-concave 
 if the function $f^{1/s}$ is concave on its support. It is well known that for a positive integer  $s,$ a $1/s$-concave function on $\R^d$ is the marginal of the indicator function of a convex set in $\R^{d+s}$.   A non-negative function on $\R^d$ is called log-concave if its logarithm is concave on its  support. 
 It is  
 also well known 
 that any log-concave function is the local uniform limit of certain $1/s$-concave functions,  as $s$ tends to infinity, e.g.,  \cite[Section 2.2]{brazitikos2014geometry}. 
 This observation has been useful in many instances in the setting of log-concave functions, e.g.,  \cite{artstein2004santalo, CaglarWerner2014, klartag2007marginals},  as it allows to pass 
 results from $1/s$-concave functions to log-concave functions.
 \par
 The concept of duality is a cornerstone of both, geometry in general, and asymptotic geometric analysis in particular. In convex analysis, the concept of duality is tightly connected with the notion of a polar set.  Recall that the polar of a set 
 $K$ in $\R^d$ is the set $\polarset{K}$ given by
 \[
 \polarset{K} = \{y \in \R^d \st \iprod{x}{y} \leq 1
  \quad \text{for all } \ x \in K \}.
 \]
 Natural generalizations of this definition to the setting of classes of concave functions are as follows.
 For any $s > 0,$ the \emph{$s$-polar transform}, introduced in  \cite{artstein2004santalo}, is defined by
\[
\slogleg{f}(y) =  \inf\limits_{\{x \st f(x)  > 0\}}
\frac{\truncf{1 - {\iprod{x}{y}}}^{s}}
{f(x)},
\]
where $\truncf{a} = \max \{a, 0\}.$
The \emph{polar} (or \emph{log-conjugate})  of a non-negative function $f$ on $\R^d$  is defined by
\[
	\slogleg[\infty]{f}(y) = 	
	\inf\limits_{\{x \st f(x) > 0\}} 
	\frac{e^{-\iprod{x}{y}}}{f(x)}.
\]
As shown in \cite[Theorem 1]{artstein2008concept}, 
the $s$-polar transform is essentially the only order reversing involution on the class of upper semi-continuous ${1}/{s}$-concave functions containing the origin in the interior of support.
Likewise,  as shown in \cite[Corollary 12]{artstein2007characterization}, $ \slogleg[\infty]{}$ is 
essentially the only order reversing involution on the class of upper semi-continuous log-concave functions. See also \cite{boroczky2008characterization}
for a similar characterization of the polar transformation in the setting of convex bodies.
\par
Alexandrov \cite{aleksandrov1967mean} noticed that the reciprocal of volume of the polar of  a convex body $K$ in $\R^d$ as a function of the center of polarity is a $1/d$-concave function on the interior of $K.$ More formally,
denote the shift of a body $K$ by  a vector $z$ by
\[
\shift{K}{z} = K - z.
\]
Then Alexandrov's results says that for a convex body 
$K \subset \R^d,$ the function   
\[
z \mapsto \parenth{\vol{d} \left(\shift{K}{z}\right)}^{-1/d}
\]
is concave on the interior of $K,$ where $\vol{d}$ is the standard $d$-dimensional volume.

In \cite{ivanov2020functional} (see also \cite{pivovarov2020stochastic}), the authors suggested a certain way of representing a $1/s$-concave function on $\R^d$ as a convex set in $\R^{d+1}.$ They called it \emph{$s$-lifting} and defined it as follows
\[
\slift{f} = 
\left\{ 
(x,\xi) \in \R^{d+1} \st x \in \closure \supp{f}, \;
 \enorm{\xi} \le 
\parenth{f(x)}^{1/s}
\right\},
\]
where $\closure \supp{f}$  denotes the closure of the support of $f.$
We note that the $s$-lifting allows us to use tools and results for convex sets in the setting of $1/s$-concave functions. The key observation of the current paper is the following simple lemma.
\begin{lem}\label{lem:dual_s_lifting}
Let $f \colon \R^d \to [0, \infty).$ Then
\[
\left(\slift{f}\right)^{\polar} = 
\slift{\left(\slogleg f \right)}.
\]
\end{lem}
 This representation allows to express  the integral of an $s$-polar function via a simple formula, even for non-integer $s$.
We elaborate on this in Section \ref{sec:slifting}.

In summary, the primary objective of this paper is to demonstrate the feasibility of leveraging tools from classical convexity alongside the examination of various classes of concave/convex functions simultaneously. Our aim is to illustrate that it is possible to derive classical results in convexity while also generating new outcomes within the functional context through specific limit arguments. However, it is important to exercise caution when employing these arguments, as, in some cases, duality, similar to what occurs in Lemma \ref{lem:dual_s_lifting}, may be lost during the process. For further insights into such instances, we direct interested readers to \cite{ivanov2023functional}.

\vskip 5mm
\subsection{The main theorems}
\vskip 3mm
\noindent
For a function $f$ on $\R^d,$ we denote its shift 
 by a vector $z$ by
\[
\shift{f}{z} (x)= f(x - z),  \hskip 3mm x \in \R^d.
\] 
The \emph{barycenter} of  an integrable function $f$ on $\R^d$ is defined as 
\[
\frac{\int_{\R^d} x f(x) \di x}{\int_{\R^d} f},
\]
if the quotient exists. Any log-concave function of finite integral has the barycenter \cite[Lemma 2.2.1]{brazitikos2014geometry}.
\vskip 2mm
A main result of this paper is the following generalization of Alexandrov's theorem. 
\vskip 2mm

\par
\noindent
\begin{thm}\label{thm:Alexandrov_s-conc}
Let $s \in (0, \infty)$ and $f : \R^d \to [0, \infty)$ be  an upper semi-continuous, $1/s$-concave function with finite integral. Then we have
\begin{enumerate}
\item\label{ass:1_thm:Alexandrov_s-conc}The function 
$z \mapsto  \int_{\R^d} \slogleg \! \parenth{\shift{f}{z}}$ 
is convex on the interior of  the support of $f$. It attains the minimum 
at point $\tilde{z}$ such that the origin is the barycenter of 
$\slogleg \! \parenth{\shift{f}{\tilde{z}}}$.
\item\label{ass:2_thm:Alexandrov_s-conc} The function 
\[
z \mapsto  
\parenth{
\int_{\R^d} \slogleg \! \parenth{\shift{f}{z}}
}^{-\frac{1}{d + s}}
\]
is concave  on the interior of the support of $f$.
\end{enumerate}
\end{thm}
\vskip 2mm
\noindent
Note that Theorem \ref{thm:Alexandrov_s-conc} yields Alexandrov's result, since
the indicator function of a convex body is $s$-concave for any positive $s$.
\par

Using a limit argument, we obtain the following version of Alexandrov's theorem for log-concave functions.
\begin{thm}\label{thm:Alexandrov_log-conc}
Let $f : \R^d \to [0, \infty)$ be an upper semi-continuous  log-concave function with finite integral. Then the function 
\[
z \mapsto  \int_{\R^d} \slogleg[\infty]\! \parenth{\shift{f}{z}}
\] 
is convex and its reciprocal is log-concave on the interior of support of $f.$
\end{thm}
 We also provide a simple direct proof of Theorem \ref{thm:Alexandrov_log-conc}
 in Subsection \ref{sec:consequences}.
 
\vskip 2mm
For any function $f$ on $\R^d,$ define 
\begin{equation} 
\label{eq:logconv_approx_s_conc}
f_s (x) = \truncf{ 1 + \frac{\log f(x)}{s}}^{s}.
\end{equation}
Clearly, for a log-concave function $f,$ the function  $f_s$ is $s$-concave and $f_s \to f$ pointwise on $\R^d$ as $s \to + \infty.$
Then, to carry out the limit argument, we use the following technical observation, which, surprisingly, seems to be new.
\par
\noindent
\begin{thm}
\label{thm:approx_by_s-conc_dual}
Let $f \colon \R^d \to [0, \infty)$ be an upper semi-continuous,  log-concave function with finite integral,  containing the origin in the interior of its support.
Denote by $A$  the set of all points in $\R^d$ that are not in the boundary of the support of $ \slogleg[\infty] f$.
Then,  as  $s \to \infty$, 
$$\slogleg f_s \!\parenth{\frac{x}{s}} \to  \slogleg[\infty] f (x)$$
locally uniformly on $A$.
\end{thm}
\vskip 2mm
\noindent
A weaker version of this result is \cite[Lemma 3.3]{artstein2004santalo}.
\vskip 3mm
One of the most important  results in convex geometry  is the Blaschke--Santal\'o inequality, see, e.g.,  \cite{GruberBook, SchneiderBook}. It says that there is a unique 
$z_0$ in $\text{int} (K)$,  the interior of $K$,  for which $\vol{d} (\polarset{\parenth{\shift{K}{z}})}$ is minimal  and then  
\[
\vol{d} (K) \cdot \vol{d} (\polarset{\parenth{\shift{K}{z_0}})} 
\leq \parenth{\vol{d} \ball{d}}^2.
\]
Equality holds if and only if $K$ is an ellipsoid.
Here, and throughout the paper,  $\ball{d}$ denotes the $d$-dimensional Euclidean unit ball.
Meyer and Pajor \cite{meyer1990blaschke} proved a more general form of the this  inequality, which we state now:
\\ \noindent
\emph{Let $K$ be a convex body in $\R^d$. Let $H$ be an affine hyperplane with half-spaces $H_{+}$ and $H_{-}$,  such that $\vol{d} \left(H_{+} \cap K\right) = \lambda \vol{d} \left(K\right)$.
 Then there exists 
$z \in H \cap \text{int} (K)$ such that \begin{equation}
\label{eq:unbalanced_santalo_sets}
\vol{d} K \cdot \vol{d} \polarset{\parenth{\shift{K}{z}}} 
\leq  \frac{\parenth{\vol{d} \ball{d}}^2}{4 \lambda(1 -\lambda)}.
\end{equation}}
\par
\noindent
We note here that $\ball{d}$ is the only self-polar set in $\R^d$.
Let $|\cdot |$ be the Euclidean norm and put
\[
\hat{h}(x) = 	\begin{cases}
			\left[1 -\enorm{x}^2\right]^{1/2},& 
		\text{ if }x \in \ball{d}\\
		0,&\text{ otherwise}.
	\end{cases}
\] 
It is not hard to see that $\hat{h}^s$ is self-$s$-polar, that is,
$$\slogleg \! \parenth{\hat{h}^s} = \hat{h}^s,$$
and thus this function  plays the role of the unit ball in the class of $1/s$-concave functions. 
Similarly, the standard Gaussian density $e^{-\enorm{x}^2/2}$ is self-polar in the class of log-concave functions.
The Blaschke--Santal\'o inequality for log-concave functions was obtained in e.g.,  \cite{artstein2004santalo, KBallthesis, lehec2009partitions}. The Blaschke--Santal\'o inequality for a more general ``duality relation'' including the $s$-polar transform  was obtained in \cite{fradelizi2007some}. In \cite{lehec2009direct}, Lehec generalized \eqref{eq:unbalanced_santalo_sets} to log-concave functions. 
\par
Using Lehec's approach, we prove the following extension of \eqref{eq:unbalanced_santalo_sets} to the class of $1/s$-concave functions.
\par
\noindent
\begin{thm} 
\label{thm:santalo_s_concave_lambda}
Let $f \colon \R^d \to [0, \infty)$ be a $1/s$-concave function with  finite integral.  Let $H$ be an affine hyperplane
with half-spaces $H_{+}$ and $H_{-}$
and such that
$\lambda \int_{\R^d} f = \int_{H_{+}} f$
for some $\lambda \in (0,1).$ 
 Then there exists 
$z \in H$ such that 
 \begin{equation}
\label{eq:s_santalo_lambda}
\int_{\R^d} f \int_{\R^d} \slogleg{\parenth{\shift{f}{z}}}  
\leq  \frac{\volbs^2}{4 \lambda(1 - \lambda)},
\end{equation}
where 
\(
\volbs = \int_{\R^d} \hat{h}^s.
\)
\end{thm}
\par
\noindent
Again, this theorem implies both the analogous result \eqref{eq:unbalanced_santalo_sets} for convex sets and Lehec's result \cite{lehec2009direct} for log-concave functions.
A kind of reverse Santal\'o inequality, that can be considered a counterpart to Theorem \ref{thm:santalo_s_concave_lambda},  was proved in \cite{MeyerReisner} for the case  $d=1$. 
See also \cite{FradeliziMeyer} for another proof of this result.
\vskip 2mm
Finally, generalizing the definition of Meyer and Werner \cite{meyer1998santalo},
we define and list several properties of the Santal\'o $s$-region of a non-negative function with bounded support. This region is essentially the set of points $z$ such that the integral of the $s$-polar transform of $\shift{f}{s}$ is bounded by some positive constant. We give formal definitions and  discuss possible definitions of Santal\'o $s$-function in Section \ref{sec:santalo_func}.
\vskip 2mm
The rest of the paper is organized as follows.
In  \ref{sec:slifting}, we define the $s$-lifting of a function, study its properties
related to duality, and prove Theorem \ref{thm:Alexandrov_s-conc}. In Section \ref{sec:log_conc_as_limit}, we recall several definitions of convex analysis and prove Theorem \ref{thm:approx_by_s-conc_dual}. We also  show that Theorem \ref{thm:Alexandrov_log-conc} is a consequence of Theorems  \ref{thm:Alexandrov_s-conc} and \ref{thm:approx_by_s-conc_dual}. Next, we prove Theorem \ref{thm:santalo_s_concave_lambda} in Section \ref{sec:santalo_ineq_s_conc}. Finally, we introduce and study the Santal\'o $s$-region in Section \ref{sec:santalo_func}.
\subsection*{Notation}
We use $\iprod{x}{y}$ to denote the standard inner product of vectors $x$ and
$y$ of $\R^d$.  We write  $\R^d \subset \R^{d+1}$,  when we consider $\R^d$ as the subspace of $\R^{d+1}$ of vectors with zero last coordinates.
We say that a set $K \subset \R^{d+1}$ is \emph{$d$-symmetric} if $K$ is symmetric with respect to $\R^d \subset \R^{d+1}$.
The \emph{closure} of a set $K \subset \R^d$ is denoted by $\closure K$. The \emph{support function} of a convex body $K$ at $y\neq 0$ is defined by
$$
\suppfunc{K}{y} = \sup_{x \in K} \langle x, y \rangle.
$$
The \emph{convex hull} of a set $K$ is denoted by  $\conv K.$ 
The \emph{support} of a non-negative function $f$ on $\R^d$ is the set
\[
\supp f = \{x \in \R^d \st f(x) > 0\}.
\]
We will refer to an integrable lower semi-continuous function of finite integral as  a  \emph{proper} function.
We will integrate over domains in $\R^d$ and $\R^{d+1}$  and 
use  $\lambda_{d+1}$ to denote the standard Lebesgue measure on $\R^{d+1}$ and $\sigma$ to denote the uniform measure on the unit sphere $S^d \subset \R^{d+1}.$

\section{The \tpdfs-lifting and its properties}
\vskip 2mm
\label{sec:slifting}
\vskip 2mm
\noindent
The notion of $s$-lifting of a function was first introduced in \cite{ivanov2020functional}. 
We recall its definition  and prove Theorem \ref{thm:Alexandrov_s-conc}
in the current section.
\vskip 2mm
Let $f \colon \R^d \to [0, \infty)$ be a function and $s>0$. The \emph{$s$-lifting} of 
$f$ is a 
$d$-symmetric set in $\R^{d+1}$ defined by
\begin{equation}\label{sLift}
\slift{f} = 
\left\{ 
(x,\xi) \in \R^{d+1} \st x \in \closure \supp{f}, \;
 \enorm{\xi} \le 
\parenth{f(x)}^{1/s}
\right\}.
\end{equation}
\par
\noindent
Clearly, the $s$-lifting of a $1/s$-concave function is a convex set. 
Thus, the $s$-lifting gives a nice way of representing a $1/s$-concave function on $\R^d$ as a convex set in $\R^{d+1}$. In fact, 
it is represented as a $d$-symmetric convex set. 
There are other representations of $1/s$-concave function as convex sets. One of them is mentioned in (\ref{def.Ksf}).
The advantages of the representation (\ref{sLift}) are 

\begin{itemize}
\item it holds  for non-integer $s$
\par
\item  one can investigate the properties of the $s$-lifting instead of studying $1/s$-concave functions directly.  
\end{itemize}
\vskip 2mm
\noindent

\begin{proof}[Proof of Lemma \ref{lem:dual_s_lifting}]
Let $y \in \R^d$ and $\tau \in \R.$
  Then  $(y, \tau) \in \left(\slift{f}\right)^{\polar}$ if and only if the inequality
  \[
  \iprod{x}{y} + t \tau \leq 1
  \]
 holds for all $x \in \supp{f}$ and $t \in \R$ such that
 $(x, t) \in \slift{f}.$  
By the symmetry of the $s$-lifting, we conclude that 
 $(y, \tau) \in \left(\slift{f}\right)^{\polar}$ if and only if 
 \[
 \iprod{x}{y} + \enorm{\tau} f^{1/s}(x) \leq 1
 \] 
for any $x \in \supp f.$
Thus,  $\iprod{x}{y} \leq 1$ and 
\[
\enorm{\tau}^{s} \leq 
\frac{\left(1 - \iprod{x}{y}\right)^s}{f(x)}.
\]
Taking the infimum, we see that
 $(y, \tau) \in \left(\slift{f}\right)^{\polar}$ if and only if 
\[
\enorm{\tau}^{s} \leq  \inf\limits_{x \in \supp{f}}
\frac{\left(1 - \iprod{x}{y}\right)^s}{f(x)} = \slogleg{f}(  y).
\]
That is, $(y, \tau) \in \left(\slift{f}\right)^{\polar}$ if and only if  $(y, \tau)$ belongs to the $s$-lifting of $\slogleg{f}.$
\end{proof}
\vskip 3mm
\noindent
Let ${C}\subset\R^{d+1}$ be a $d$-symmetric Borel set. 
The \emph{$s$-volume} of ${C}$ is defined by
\[
\smeasure{{C}}=\int_{\R^d} 
\left[\frac{1}{2}\length{{C}\cap\ell_x}\right]^s \di x.
\] 
By Fubini's theorem, we have
\begin{equation}
\label{eq:integral_slift_repr}
\int_{\R^d} f =  \smeasure{\slift{f}} = 
\frac{s}{2}\int_{\slift{f}} \abs{\iprod{e_{d+1}}{x}}^{s-1} \di \lambda_{d+1}(x).
\end{equation}
\vskip 3mm
\begin{lem}\label{lem:int_s-polar_via_suppfunc}
Fix $s > 0.$
Let $K$ be a $d$-symmetric convex body in $\R^{d+1}.$ The functional
\[
z \mapsto
\frac{s}{2(d+s)}\int_{S^{d}} 
\frac{\abs{\iprod{e_{d+1}}{ u}}^{s-1}}{
\left(h_{\shift{K}{z}}(u)\right)^{d+s}} \di \spherem{u}
\]
is convex on the interior of $K.$  Moreover, if $K$ is the $s$-lifting of 
a $1/s$-concave function $f$ and $z \in \R^d,$ then the value of this functional at $z$ is equal to 
\[
\int_{\R^d} \slogleg{\!\parenth{\shift{f}{z}}}.
\]
\end{lem}
\begin{proof}
Since $\suppfunc{\shift{K}{z}}{u} = \suppfunc{K}{u} - \iprod{z}{u},$ 
one sees that $\parenth{\suppfunc{\shift{K}{z}}{u}}^{d+s}$ is a convex function of $z$ for a fixed $u \in S^{d}$ and any $s > 0.$
The convexity of the functional  follows immediately.
\par
\noindent
Assume now that $z \in \R^d$ and put $K = \slift{f}$ for a $1/s$-concave function $f$.
By Lemma \ref{lem:dual_s_lifting} and with  equation \eqref{eq:integral_slift_repr}, we get
\[
\int_{\R^d} \slogleg \!\parenth{\shift{f}{z}}  = 
\smeasure{\polarset{\shift{K}{z}}}=
\frac{s}{2}\int_{\polarset{\shift{K}{z}}} \abs{\iprod{e_{d+1}}{x}}^{s-1} \di \la_{d+1}(x). 
\]
Using spherical coordinates gives
\[
\int_{\R^d}  \slogleg\!\parenth{ \shift{f}{z}}   =
\frac{s}{2}\int_{\polarset{\shift{K}{z}}} \abs{\iprod{e_{d+1}}{r u}}^{s-1} r^{d}\di r \di \spherem{u} =
 \]
\[
\frac{s}{2}\int_{S^{d}} 
\abs{\iprod{e_{d+1}}{ u}}^{s-1} 
\int_{\left[0, \frac{1}{\suppfunc{\shift{K}{z}}{u}}\right]}  r^{d+s-1}\di r \di \spherem{u} .
\]
That is,
\begin{equation}
\label{eq:integ_stransf_rep}
\int_{\R^d}  \slogleg \!\parenth{\shift{f}{z}}  = 
\frac{s}{2(d+s)}\int_{S^{d}} 
\frac{\abs{\iprod{e_{d+1}}{ u}}^{s-1}}{
\parenth{\suppfunc{\shift{K}{z}}{u}}^{d+s}} \di \spherem{u}.
\end{equation}
This completes the proof.
\end{proof}

\vskip 3mm
\noindent
\begin{proof}[Proof of Theorem \ref{thm:Alexandrov_s-conc}]
Denote $K_z = \slift{\parenth{\shift{f}{z}}},$
$K^\polar_{z} = \slift{\slogleg \! \parenth{\shift{f}{z}}},$ and
\[
\Phi(z) =  \int_{\R^d}  \slogleg \!\parenth{\shift{f}{z}} .
\]
Lemma \ref{lem:int_s-polar_via_suppfunc} implies that  $\Phi$ 
is convex on the interior of the support of $f.$
\par
\noindent
We now address the second assertion of the theorem.
Denote $\Psi(z)  = \Phi(z)^{-\frac{1}{d+s}},$
and let $\nu_z $ be a measure on $S^d$ with density given by
\[
\di \nu_z(u) = 
\frac{s}{2(d+s)}
\frac{\abs{\iprod{e_{d+1}}{ u}}^{s-1}}{
\parenth{\suppfunc{K_z}{u}}^{d+s}} \di \spherem{u}.
\]
The directional derivative of
$\suppfunc{K_z}{u}$ in the direction of the $i$-th standard basis vector $e_i$ is  
$ -u[i],$ where $a[i]$  stands for the $i$-th coordinate of $a.$ 
Differentiating \eqref{eq:integ_stransf_rep},  we obtain that
\[
\Phi^\prime_{e_i}(z) = 
(d + s)  \frac{s}{2(d+s)}\int_{S^{d}} 
\frac{\abs{\iprod{e_{d+1}}{ u}}^{s-1} u [i]}{
\parenth{\suppfunc{K_z}{u} }^{d+s+1} } \di \spherem{u} = 
(d + s) \int_{S^{d}} 
\frac{u [i]}{
\suppfunc{K_z}{u}} \di \nu_z(u),
\]
and 
\[
\Phi^{\prime\prime}_{e_i e_j}(z) =
(d+s)(d+s+1)
\int_{S^{d}} 
\frac{u [i] u [j]}{
\parenth{\suppfunc{K_z}{u}}^2} \di \nu_z(u).
\]
On the other hand, 
\begin{eqnarray*}
\Phi^{\prime\prime}_{e_i e_j}(z) &=&
\frac{1}{d+s} \parenth{\frac{d+s+1}{d+s}} \Phi(z)^{-\frac{1}{d+s} - 2}  \Phi^\prime_{e_i}(z) 
\Phi^\prime_{e_j}(z) - 
\frac{1}{d+s} \Phi(z)^{-\frac{1}{d+s} - 1}  \Phi^{\prime\prime}_{e_i e_j}(z)\\
&=&
 - \frac{1}{d+s} \Phi(z)^{-\frac{1}{d+s} - 2}
\parenth{
\Phi(z)  \Phi^{\prime\prime}_{e_i e_j}(z) -
\frac{d+s+1}{d+s} \Phi^\prime_{e_i}(z) 
\Phi^\prime_{e_j}(z)
}.
\end{eqnarray*}
Therefore, $\Psi$ is concave on its support if and only if
the matrix $A_z$ given by
\[
A_z[i,j] = \Phi(z)  \Phi^{\prime\prime}_{e_i e_j}(z) -
\frac{d+s+1}{d+s}  \Phi^\prime_{e_i}(z) 
\Phi^\prime_{e_j}(z)
\]
is positive semi-definite at every point of the interior of the  support of $\Psi.$
Using the formulas for the partial derivatives of $\Phi$ and \eqref{eq:integ_stransf_rep}, we get
\begin{eqnarray*}
&&\frac{A_z[i,j]}{(d+s)(d+s+1)} =  \\
&& \int_{S^{d}} 
 \di \nu_z(u) \cdot \int_{S^{d}}  
\frac{u [i] u [j]}{
\parenth{\suppfunc{K_z}{u}}^2} \di \nu_z(u) 
-
 \int_{S^{d}} 
\frac{u [i]}{
\suppfunc{K_z}{u}} \di \nu_z(u)
 \int_{S^{d}} 
\frac{u [j]}{
\suppfunc{K_z}{u}} \di \nu_z(u).
\end{eqnarray*}
That is, $A_z$ is a covariance matrix, and hence it is positive semi-definite. 
\par
\noindent
Finally, differentiating  \eqref{eq:integ_stransf_rep} again, we get that the directional derivative $\Phi^\prime_{h}$ of $\Phi$ in direction $h \in \R^d$ is
\[
\Phi^\prime_{h} (z) = 
 \frac{s}{2}\int_{S^{d}} 
\frac{\abs{\iprod{e_{d+1}}{ u}}^{s-1} }{
\parenth{\suppfunc{K_z}{u} }^{d+s+1} }  \iprod{h}{u}\di \spherem{u} .
\]
Reversing the chain of identities in the proof of 
Lemma \ref{lem:int_s-polar_via_suppfunc}, one gets
\[
(d+s+1) \Phi^\prime_{h} (z) = 
\frac{s}{2} \int_{S^{d}} 
\abs{\iprod{e_{d+1}}{ u}}^{s-1}  \iprod{h}{u}
\int_{\left[0, \frac{1}{\suppfunc{K_z}{u}}\right]}  r^{d+s} \di r \di \spherem{u} =
\]
\[
\frac{s}{2} \int_{K^\polar_{z}}
 \abs{\iprod{e_{d+1}}{r u}}^{s-1}  \iprod{h}{r u} r^{d}\di r \di \spherem{u} =
\frac{s}{2} \int_{K^\polar_{z}} \abs{\iprod{e_{d+1}}{x}}^{s-1}
 \iprod{h}{x} \di \la_{d+1}(x).
 \]
 Since $h \in \R^d$ and by the definition of $K^\polar_{z},$ the latter is
 \[
  \int_{\R^d} \iprod{h}{y} \slogleg \! \parenth{\shift{f}{z}} (y) \di y.
 \]
 By convexity, all directional derivatives of $\Phi$ vanish at the argmin. Consequently, the above calculations show that at the argmin $\tilde{z}$, the identity 
 \[
  \int_{\R^d} {y} \slogleg \! \parenth{\shift{f}{\tilde{z}}} (y) \di y =0
 \]
holds. Thus, the origin is the barycenter of $\slogleg \! \parenth{\shift{f}{\tilde{z}}}.$
\vskip 2mm
\noindent
This finishes the proof of Theorem \ref{thm:Alexandrov_s-conc}.

\end{proof}
As was pointed to us by the anonymous referee,  using the measure representation of the integral from Lemma \ref{lem:int_s-polar_via_suppfunc}, combined with the appropriate version of Minkowski inequality for means yield the inequality of Theorem \ref{thm:Alexandrov_s-conc}. Let us sketch the idea. 

\vskip 2mm
\noindent

In the notation used in the proof of Theorem \ref{thm:Alexandrov_s-conc},
\[
\Psi(z) = \norm{h_K - \iprod{z}{\cdot}}_{L_{-(d+s)} (\tilde{\nu}_s)},
\]
where $\tilde{\nu}_s $  is a measure on $S^d$ with density given by
\[
\di \tilde{\nu}_s(u) = 
\frac{s}{2(d+s)}
{\abs{\iprod{e_{d+1}}{ u}}^{s-1}} \di \spherem{u}.
\]
Using the Minkowski inequality for means, one gets:
\[
\Psi(\lambda z_1 + (1- \lambda) z_2) = 
 \norm{\lambda(h_K - \iprod{z_1}{\cdot}) + (1- \lambda)(h_K - \iprod{z_2}{\cdot})}_{L_{-(d+s)} (\tilde{\nu}_s)} \geq 
\]
\[
\lambda   \norm{h_K - \iprod{z_1}{\cdot}}_{L_{-(d+s)} (\tilde{\nu}_s)} +
(1- \lambda)  \norm{h_K - \iprod{z_2}{\cdot}}_{L_{-(d+s)} (\tilde{\nu}_s)} = \lambda \Psi(z_1) + (1- \lambda)\Psi(z_2).
\]
\vskip 10mm

\section{Log-concave functions}
\label{sec:log_conc_as_limit}
In this section, we study the properties of log-concave functions related to the polar transform and prove Theorems \ref{thm:approx_by_s-conc_dual} and \ref{thm:Alexandrov_log-conc}.  

\vskip 2mm
\noindent
\subsection{Consequences of Theorem \ref{thm:approx_by_s-conc_dual}}
\label{sec:consequences}
\par
\noindent
Before proving Theorem \ref{thm:approx_by_s-conc_dual}, we derive several results from it including Theorem \ref{thm:Alexandrov_log-conc}.
\par
\noindent
\begin{proof}[Proof of Theorem  \ref{thm:Alexandrov_log-conc}]
Recall that
\begin{equation}
\label{eq:s-to-zero_limit}
\lim\limits_{s \to + \infty}
\parenth{\lambda a^{\frac{1}{d+s}} + (1- \lambda) b^{\frac{1}{d+s}}}^{d+s} =
a^{\lambda}b^{1-\lambda}
\end{equation}
for any positive real numbers $a$ and $b$. 
\par
\noindent 
Let now $f$ be a proper log-concave function  containing the origin in the interior of its support.
Then, to prove Theorem \ref{thm:Alexandrov_log-conc},  it suffices to show that
\begin{equation}
\label{eq:log-conv_approx_conv_integ}
\int_{\R^d} \slogleg[\infty] f  = \lim\limits_{s \to + \infty} s^d \int_{\R^d} \slogleg f_s ,
\end{equation}
where $f_s$ is as in (\ref{eq:logconv_approx_s_conc}).
Indeed, the convexity of $z \mapsto  \int_{\R^d} \slogleg[\infty]\! \parenth{\shift{f}{z}}$ follows immediately from  assertion (\ref{ass:1_thm:Alexandrov_s-conc}) of Theorem \ref{thm:Alexandrov_s-conc} and \eqref{eq:log-conv_approx_conv_integ}. The log-concavity of 
\[
z \mapsto  \frac{1}{\int_{\R^d} \slogleg[\infty]\! \parenth{\shift{f}{z}}}
\]
 follows  from  assertion (\ref{ass:2_thm:Alexandrov_s-conc}) of Theorem \ref{thm:Alexandrov_s-conc} and  identity \eqref{eq:s-to-zero_limit}.
\par
\noindent
Identity (\ref{eq:log-conv_approx_conv_integ}) follows from Theorem \ref{thm:approx_by_s-conc_dual} and 
\cite[Lemma 3.2]{artstein2004santalo}. We use the following  simplified version of this lemma.
\vskip 2mm
\noindent
\begin{lem}\cite{artstein2004santalo} \label{lem:convergence_log-conc}
Let $\{f_i\}_1^{\infty}$ be a sequence of log-concave functions converging to a log-concave function $f$ of finite integral on a dense subset $A \subset \R^d.$
Then $\int_{\R^d} f_n \to \int_{\R^d} f.$ 
\end{lem}
\par
\noindent
This completes the proof of Theorem \ref{thm:Alexandrov_log-conc}.
\end{proof}

Also, there is a direct proof of Theorem \ref{thm:Alexandrov_log-conc}.
\par
\noindent
\begin{proof}[Direct proof of Theorem  \ref{thm:Alexandrov_log-conc}]
Denote 
$\Phi(z) =  \int_{\R^d} \slogleg[\infty]\! \parenth{\shift{f}{z}},$
and let $f = e^{-\psi}.$
One has 
\[
\Phi(z) =  \int_{\R^d} e^{- \slogleg[]\psi(y) - \iprod{z}{y}}  \di y.
\]
By the H\"older inequality,
\[\Phi\! \parenth{(1-\lambda) z_1+\lambda z_2} = 
%\int_{\R^d} e^{- \slogleg[]\psi(y) - \iprod{(1-\lambda) z_1+\lambda z_2}{y}}  \di y =
\int_{\R^d} e^{- (1-\lambda)\left( \slogleg[]\varphi(y)-
 \iprod{z_1}{y}\right)
 -\lambda\left(\slogleg[]\varphi(y)- \iprod{z_2}{y}\right)} \di y 
\]
\[
\leq 
\parenth{\int_{\R^d} e^{- \slogleg[]\varphi(y)-
 \iprod{z_1}{y}} \di y}^{1-\lambda}
\parenth{ \int_{\R^d} e^{-\slogleg[]\varphi(y)- \iprod{z_2}{y}} \di y}^{\lambda} = 
\Phi(z_1)^{1 - \lambda} \Phi(z_2)^{\lambda}.
\]
So, the function $\Phi$ is log-convex. Consequently, it is convex and the desired functional, which is reciprocal of $\Phi,$ is log-concave.

%
%\slogleg[\infty]\! \parenth{\shift{f}{z}}
%
\end{proof}
\vskip 3mm
\noindent
The next  corollary  is another consequence of Theorem \ref{thm:approx_by_s-conc_dual}.
\vskip 2mm
\noindent
\begin{cor}\label{cor:approx_mahler_int}
Let $f : \R^d \to [0, \infty)$ be a proper log-concave function containing the origin in the interior of its support. Then
\[
\int f  \cdot \int \slogleg[\infty] f  = \lim\limits_{s \to + \infty}
s^d \int f_s \cdot \int \slogleg f_s.
\]
\end{cor}
\vskip 5mm
\subsection{Proof of Theorem 3}

To prove  Theorem \ref{thm:approx_by_s-conc_dual}, we recall several  definitions and facts of convex analysis which can be found in e.g., \cite{rockafellar1970convex}.
\vskip 2mm
\noindent
We start with the definition of the classical \emph{convex conjugate transform} or 
 \emph{Legendre transform} $\legendre$ defined for functions $\phi: \R^d \to \R\cup \{+\infty\}$ by
\[
	(\slogleg[]{\phi})(y) = \sup\limits_{x \in \R^d} \{\iprod{x}{y} - \phi(x)\}.
\]
Thus,  for  $f = e^{-\psi} : \R^d \to [0, \infty)$, 
\[
	\slogleg[\infty]{f}(y) = e^{- (\slogleg[] \psi)(y)}.
\]
\par
\noindent
A vector $p$ is said to be a \emph{subgradient} of a convex function 
$\psi$ on $\R^d$ at the point $x$ if
\[
\psi(y) \geq \psi(x) + \iprod{p}{y-x}
\]
for all $y \in \R^d$.
The set of all subgradients of $\psi$ at $x$ is called the \emph{subdifferential}
of $\psi$ at $x$ and is denoted by $\partial \psi(x)$ 
\par
\noindent
The \emph{effective domain}  $\dom \psi$ of a convex function $\psi$ on $\R^d$ is the set 
\[
\dom \psi = \left\{x \st  \psi(x) < + \infty\right\}.
\]
The \emph{epigraph} of a convex function $\psi$ on $\R^d$ is the set
\[
\left\{(x, \xi) \st x\in \dom \psi, \; \xi \in \R, \;    \xi \geq \psi(x) \right\}.
\]
\vskip 3mm
\noindent
In the remainder of the paper we work with convex functions that have  non-empty effective domain and closed epigraph.
\vskip 2mm
\noindent
The following statement is a direct consequence of the definition of subdifferential.
\vskip 2mm
\noindent
\begin{lem}[Geometric meaning of subdifferential]
\label{lem:geom_meaning_of_subdif}
Let $\phi: \R^d \to \R \cup \{+\infty\}$ be a lower semi-continuous, convex function with non-empty effective domain. 
Let $x \in \dom \phi$.
If $p \in \R^d$ and a negative number $\xi$ are such that
\[
 \iprod{(p, \xi)}{(y, \phi(y))} \leq 
  \iprod{(p, \xi)}{(x, \phi(x))}
\]
for all $y \in \dom \psi,$ then 
there are a subgradient $q$ of $\phi$ at $x$ and a positive constant $\alpha$
such that
\[
(p, \xi) = \alpha (q,-1).
\]
\end{lem}
\vskip 2mm
\noindent
\begin{rem}
The  assertion of Lemma \ref{lem:geom_meaning_of_subdif} can be rephrased as:  $(p, \xi)$ belongs to the \emph{normal cone} to the epigraph of $\psi$ at point $(x, \psi(x))$.
\end{rem}
\vskip 3mm
\noindent
The following facts on the  subgradient  can be found in e.g., 
\cite[Chapter 23]{rockafellar1970convex}. 
See also \cite{clarke1990optimization}.
\vskip 2mm
\noindent
Let $\psi: \R^d \to \R \cup \{+\infty\}$ be a lower semi-continuous convex function with non-empty effective domain, and let 
$z$ be in the interior of $\dom \slogleg[]{\psi}$. 
Let $\epsilon >0$ be such that  $z + \epsilon \, \ball{d}_2 \subset \dom{\slogleg[]{\psi}}$. 
Then $\slogleg[]{\psi}$ is Lipschitz on $z + \epsilon\,  \ball{d}_2$ and, denoting  the Lipschitz constant by $C$, 
 the following statements hold.
\begin{enumerate}
 \item Fact 1: 
 $\partial \slogleg[] \psi (z)$ is a non-empty compact convex set and
 $\partial \slogleg[] \psi (z) \subset C \ball{d}_2$. 
 \vskip 2mm
 \item Fact 2: 
 Let $q  \in \partial \slogleg[]\psi(z).$ Then $z \in \partial \psi(q),$ and
 \begin{equation}
\label{eq:fenchel_moro_lem}
\psi(q) + \slogleg[]\psi(z) = \iprod{q}{z}
\end{equation}
\vskip 2mm
 \item Fact 3: 
 Let $q  \in \partial \slogleg[]\psi(z).$ Then, by the previous assertion,
 \begin{equation}
 \label{eq:dual_subgrad_norm}
 \enorm{\psi(q)} \leq \enorm{\slogleg[] \psi(z)} + C\enorm{z}.
 \end{equation}
\vskip 2mm
 \item Fact 4: 
Let $q  \in \partial \slogleg[]\psi(z)$. 
Then for all $p \in \R^d$, 
 \begin{equation}
\label{eq:conjugate_func_ineq} 
\iprod{z}{p} - \psi(p) \leq \slogleg[]\psi(z).
\end{equation}
 \end{enumerate} 
\vskip 3mm
\noindent
To prove Theorem \ref{thm:approx_by_s-conc_dual}, we will also need the following lemmas.
\vskip 2mm
\begin{lem}\label{lem:point_in_supp_L_infty}
Let $f \colon \R^d \to [0, \infty)$ be a proper log-concave function.
Let $z$ be a point in the interior of $\supp \slogleg[\infty] f.$
Let $C$ be the Lipschitz constant of the convex function 
$\slogleg[]\psi = -\log \slogleg[\infty] f$ 
on some open neighborhood of $z.$
Then for any $s >  \enorm{\slogleg[] \psi(z)} + C\enorm{z},$
\begin{equation}
\label{eq:spolar_sfunc_value}
\slogleg f_s \parenth{\frac{z}{\slogleg[]\psi(z) + s}} = 
\parenth{{1+ \frac{\slogleg[]\psi(z)}{s}}}^{-s}.
\end{equation}
\end{lem}
\vskip 2mm
\noindent
\begin{proof}
Denote $\psi = - \log f$. By the above Fact 1, there is $q \in \partial \slogleg[]\psi(z)$.
By \eqref{eq:dual_subgrad_norm}, 
\[
f_s^{1/s}(q) = \truncf{1 - \frac{\psi(q)}{s}} = 1 - \frac{\psi(q)}{s} > 0.
\]
Using \eqref{eq:conjugate_func_ineq} and \eqref{eq:fenchel_moro_lem}, one gets for all $p \in \supp{f_s}$ 
\[
\iprod{ \parenth{p, {f_s(p)^{1/s}}}}{\parenth{\frac{z}{s},1}} = 
1+ \frac{\iprod{z}{p} - \psi(p)}{s} \leq 1+ \frac{\slogleg[]\psi(z)}{s}.
\]
Since equality in the rightmost inequality is achieved at $p =q$ and since $1+ \frac{\slogleg[]\psi(z)}{s} > 0$ by the choice of $s$, 
we conclude that the vector 
\[
\frac{1}{1+ \frac{\slogleg[]\psi(z)}{s}} \parenth{\frac{z}{s},1} = 
\parenth{\frac{z}{\slogleg[](z) + s}, \frac{1}{1+ \frac{\slogleg[]\psi(z)}{s}}}
\]
belongs to the boundary of $\parenth{\slift{f_s}}^\circ$.
Thus, \eqref{eq:spolar_sfunc_value} follows from Lemma \ref{lem:dual_s_lifting}.
\end{proof}
\vskip 3mm
\noindent
\begin{lem}\label{lem:s_dual_point_representation}
Let $f=e^{-\psi} \colon \R^d \to [0, \infty)$ be a proper log-concave function such that the origin is in the interior of $ \supp   f_s.$
Denote $\slogleg[]\psi = - \log\slogleg[\infty] f.$ 
If $z$ is in the interior of $\supp \slogleg{f_s},$
then there exists $z_s$ in  the support of $\slogleg[\infty]{f}$
such that
\begin{equation}
\label{eq:spolar_sfunc_value2}
z= \frac{z_s }{\slogleg[]\psi(z_s) + s}
\quad \text{and} \quad 
\slogleg f_s \parenth{{z}} =
\parenth{{1+ \frac{\slogleg[]\psi(z_s)}{s}}}^{-s}.
\end{equation}
\end{lem}
\vskip 2mm
\noindent
\begin{proof}
Let $\parenth{y, f_s^{1/s}(y)}$ be a dual point to the convex set 
$\slift{\slogleg f_s}$ at 
$\parenth{z, \parenth{\slogleg{f_s}}^{1/s} \parenth{{z}}},$ that is
\begin{equation}
\label{eq:s_dual_point_rep1}
\iprod{\parenth{y, f_s^{1/s}(y)}}{\parenth{\tilde{z}, \parenth{\slogleg{f_s}}^{1/s} \parenth{{\tilde{z}}}}} \leq 
\iprod{\parenth{y, f_s^{1/s}(y)}}{\parenth{z, \parenth{\slogleg{f_s}}^{1/s} \parenth{z}}} = 1
\end{equation}
for all $\tilde{z} \in \supp \slogleg f_s$ and 
\begin{equation}
\label{eq:s_dual_point_rep2}
\iprod{\parenth{\tilde{y}, f_s^{1/s}(\tilde{y})}}
{\parenth{{z}, \parenth{\slogleg{f_s}}^{1/s} \parenth{z}}} 
\leq 
\iprod{\parenth{y, f_s^{1/s}(y)}}{\parenth{{z}, \parenth{\slogleg{f_s}}^{1/s} \parenth{{z}}}} = 1
\end{equation} 
for all $\tilde{y} \in \supp f_s$.
\par
\noindent
Since $z$ 
is in the interior of $\supp \slogleg f_s, $ inequality \eqref{eq:s_dual_point_rep1} implies that
$f_s(y) > 0$. 
Note that $y$ might belong to the boundary of the support of $f_s$.
Consider the convex function
\[
\phi(x) = \begin{cases}
-  f_s^{1/s}(x), & x \in \closure \supp f_s \\
+\infty, &  x \notin \closure \supp f_s.
\end{cases}
\]
Using inequality \eqref{eq:s_dual_point_rep2} in Lemma \ref{lem:geom_meaning_of_subdif}, we conclude that
\[
\parenth{z, \parenth{\slogleg{f_s}}^{1/s} \parenth{z}} = \alpha 
\parenth{\frac{z_s}{s},1},
\]
where $\alpha > 0$ and $\frac{z_s}{s}$ belongs to the subdifferential of the convex function 
$-f_{s}^{1/s}$ at $y$.
Since $f_s^{1/s}(y) > 0$ and $f_{s}$ is lower semi-continuous, 
there is an open  neighborhood $U$ of $y$ such that for all 
$\tilde{y}$ that are in $U$ and in  the boundary of 
$\closure \supp f_{s}$ we have  $f_{s}(\tilde{y})> 0$.
Moreover, since $\psi(y) < s$ and $\psi$ is upper semi-continuous,  one has 
that $\psi(\tilde{y}) = + \infty$ for all 
$\tilde{y} \in \parenth{y+ \epsilon\, \ball{d}_2} \cap \parenth{\R^d \setminus \supp f_{s}}$
for some $\epsilon > 0.$ That is, the function $\phi$ coincides with $ -1 + \frac{\psi(x)}{s}$ on some open neighborhood of $y.$ 
Hence,   $z_s \in \partial \psi(y).$
Therefore, by \eqref{eq:fenchel_moro_lem},
\[ 
\iprod{\parenth{y, f_s^{1/s}(y)}}{\parenth{\frac{z_s}{s},1} } = 
1+ \frac{\iprod{z_s}{y} - \psi(y)}{s} =  1+ \frac{\slogleg[]\psi(z_s)}{s}  < + \infty,
\]
and the identities  \eqref{eq:spolar_sfunc_value2} hold.
\end{proof}
\vskip 3mm
The following  corollary is an immediate consequence of of Lemma \ref{lem:s_dual_point_representation}.
\vskip 2mm
\noindent
\begin{cor}
\label{cor:support_of_slogleg}
Let $f \colon \R^d \to [0, \infty)$ be a proper log-concave function containing the origin  in it's support.
Denote $\slogleg[]\psi = -\log \slogleg[\infty] f$ and $M = \min_{\R^d} \slogleg[]\psi.$ 
Then for any $s$  such that the origin is in the interior of $\supp f_s$ and $s + M > 0,$
one has
\[
\supp \slogleg f_s \subset \frac{1}{M+s} \closure \supp \slogleg[\infty] f.
\] 
\end{cor}
\vskip 3mm
We are now ready for the proof of Theorem \ref{thm:approx_by_s-conc_dual}.
\vskip 2mm
\noindent
\begin{proof}[Proof of Theorem \ref{thm:approx_by_s-conc_dual}]
We put  $\psi = - \log f$. Then  $\slogleg[] \psi = - \log \slogleg[\infty] f$.
\vskip 2mm
\noindent
Corollary \ref{cor:support_of_slogleg} shows that it suffices to consider points 
 in the interior of the support of  $\slogleg[\infty] f$.
\par
Let therefore $x$ be a point in the interior of the support of  $\slogleg[\infty] f$.
That is, $x$ is  an  interior point of $\dom \slogleg[]\psi$. Let $ \epsilon_1$ be such that 
$x + \epsilon_1 \ball{d}_2 \subset \supp \slogleg[\infty] f$.
As noted above, $\slogleg[\infty] f$ is then Lipschitz on  some open neighborhood in $x + \epsilon_1 \ball{d}_2$ with Lipschitz  constant $C$. 
By Lemma \ref{lem:point_in_supp_L_infty}, we have for all sufficiently large $s$, 
\[
\left\{
 \parenth{\frac{z}{\slogleg[]\psi(z) + s}}
\st 
z \in x + \epsilon_1 \ball{d}_2
\right\} \subset \supp \slogleg f_s.
\]
\par
\noindent
Denote by $\ell$ the line passing through the origin and $x$.  If $x$ and the origin coincide, then $\ell$ is an arbitrary linear one-dimensional subspace of $\R^d$.
By continuity, for any $0 < \epsilon_2 <1$ there exists $s_0$ such that 
\[
x \in  \left\{\frac{s\, z}{\slogleg[]\psi(z) + s} \st z \in \ell, \; \enorm{z-x} < \epsilon_2 \right\}
\]
for all $s > s_0$. It follows that for all sufficiently large $s$ 
there exists $z_s$ satisfying
\begin{enumerate}
\item $\frac{s\, z_s}{\slogleg[]\psi(z_s) + s} = x$
\vskip 2mm
\item $z_s \to x$ as $s \to \infty.$
\end{enumerate}
This, continuity and identity \eqref{eq:spolar_sfunc_value} yield
\[
\lim \limits_{s \to \infty} \slogleg f_s \parenth{ \frac{x}{s}} = 
\lim \limits_{s \to \infty}
\slogleg f_s \parenth{\frac{z_s}{\slogleg[]\psi(z_s) + s}} = 
\lim \limits_{s \to +\infty}
\parenth{{1+ \frac{\slogleg[]\psi(z_s)}{s}}}^{-s} =
\slogleg[\infty] f \parenth{x}.
\]
Since all the  functions considered are continuous at any point of the interior of the support, we conclude that $f_s$ converges locally uniformly to $f$ on $A$.
\vskip 2mm
\noindent
This finishes the proof of Theorem \ref{thm:approx_by_s-conc_dual}.
\end{proof}
\vskip 10mm
\section{A Blaschke--Santal\'o inequality}
\label{sec:santalo_ineq_s_conc}
\vskip 2mm
\noindent
In this section, we prove the following theorem, which  implies
 Theorem \ref{thm:santalo_s_concave_lambda}.  Recall also that 
 \begin{equation}\label{kappa}
 \volbs = \int_{\ball{d}_2}  \left[1 -\enorm{x}^2\right]^{s/2}\,  \di x.
 \end{equation}
\vskip 2mm
\noindent
\begin{thm} 
\label{thm:santalo_s_general}
Let $f \colon \R^d \to [0, \infty)$ be 
$1/s$-concave function with finite integral.
Let $H$ be an affine hyperplane with 
half-spaces $H_{+}$ and $H_{-}$
and such that
$\lambda \int_{\R^d} f = \int_{H_{+}} f$
for some $\lambda \in (0,1)$.
 Then there exists 
$z \in H$ such that for any Borel function 
$g \colon \R^d \to [0, \infty)$ satisfying
\[
\forall x, y \in \R^d, \quad f(x + z) g(y) \leq 
\truncf{1 - {\iprod{x}{y}}}^s
\]
the inequality
\begin{equation}
\label{eq:s_santalo_gen}
\int_{\R^d} f  \int_{\R^d} g   
\leq   \frac{\volbs^2}{4 \lambda (1-\lambda)}.
\end{equation}
holds.
\end{thm}
\vskip 3mm
\noindent
The idea of our  proof mostly follows Lehec's arguments \cite{lehec2009direct} -- namely, we prove it by induction on the dimension. 
However, in our setting the one dimensional case requires a more subtle analysis of the Lebesgue level sets of the functions  than in the case of log-concave functions.
\vskip 5mm
\subsection{The one-dimensional case}
The following lemma is a particular case of Proposition 1 of \cite{fradelizi2007some}.
\begin{lem}\label{lem:s_santalo_onedim}
Let $\phi_1 \colon [0, \infty) \to [0, \infty)$ and 
$\phi_2 \colon [0, \infty) \to [0, \infty)$ be two Borel functions  satisfying the duality relation
\begin{equation}
\label{eq:onedim_santalo_lemma}
\text{for all} \;\; t_1, t_2 \in [0, \infty), \quad 
\phi_1( t_1) \phi_2( t_2) \leq 
\truncf{1 - t_1 t_2}^s .
\end{equation}
Then
\begin{equation}
\label{eq:s_santalo_lemma_onedim}
\int_{[0, \infty)} \phi_1 \int_{[0, \infty)} \phi_2 
\leq   \left(\frac{\sthing\kappa_{2}}{2}\right)^2.
\end{equation}
\end{lem}
\begin{proof}
Define $\phi_3(t) = \truncf{1 - t^2}^s;$
for $i=1,2,3,$ define $g_i (t) = \phi_i(e^t)e^t.$
On the one hand, $\int_{\R} g_i = \int_{[0, + \infty)} \phi_i.$
On the other hand, 
inequality \eqref{eq:onedim_santalo_lemma} takes the following form:
\[
g_1(t_1) g_2(t_2) \leq g_3 \parenth{\frac{t_1 + t_2}{2}}.
\]
The desired inequality follows now from the Pr\'ekopa--Leindler inequality \cite{prekopa1971logarithmic}, see also \cite{Gardner, SchneiderBook}. 
\end{proof}
\vskip 3mm
\noindent
\begin{cor}\label{cor:onedim_satalo_s_lambda}
Let $f \colon \R \to [0, \infty)$ be a  Borel function such that $\int_{\R} f \, \di t < \infty,$ and 
let 
$\int_{[0, \infty)} f  = \lambda \int_{\R} f $  
for some $\lambda \in (0,1)$.
Then for any Borel function
$g \colon \R \to [0, \infty)$ satisfying
\[
\forall t_1, t_2 \in \R, \quad f(t_1) g(t_2) \leq 
\truncf{1 - {\iprod{t_1}{t_2}}}^s
\]
the inequality
\begin{equation}
\label{eq:s_santalo_gen_onedim}
\int_{\R} f  \int_{\R} g   
\leq   \frac{\sthing\kappa^2_2}{{4 \lambda (1-\lambda)}}
\end{equation}
holds. 
\end{cor}
\begin{proof}
We use Lemma \ref{lem:s_santalo_onedim} twice. 
First with $\phi_1(t) = f(t)$ and 
$\phi_2(t) = g(t)$ and then 
with  $\phi_1(t) = f(-   t)$ and 
$\phi_2(t) = g(-  t)$.
In both cases, the condition 
$$
\phi_1( t_1)\,  \phi(t_2)  \leq \truncf{1 - t_1 t_2}^s
$$ 
for all $t_1, t_2 \in [0, \infty)$ is satisfied. Therefore, 
\[
\left(\frac{\sthing\kappa_2}{2}\right)^2 \geq 
\int_{[0,\infty)} f(t) \di t 
\int_{[0,\infty)}g(t) \di t
\]
and 
\[
\left(\frac{\sthing\kappa_2}{2}\right)^2 \geq 
\int_{(-\infty,0]} f(t) \di t 
\int_{(-\infty,0]} g( t) \di t .
\]
By the assumption $\int_{[0, \infty)} f = \lambda \int_{\R} f$,  we get 
\[
\parenth{\frac{\sthing\kappa_2}{2}}^2 \geq 
 {\lambda} \int_{\R} f( t) \di t 
\int_{[0,\infty)} g( t) \di t
\]
and
\[
\left(\frac{\sthing\kappa_2}{2}\right)^2 \geq 
 \parenth{1 - \lambda} \int_{\R} f( t) \di t 
\int_{(-\infty,0]} g( t) \di t.
\]
Summing these inequalities, we obtain
\[
 \left(\frac{\sthing\kappa_2}{2}\right)^2 \frac{1}{ \lambda(1 - \lambda)} =    \left(\frac{\sthing\kappa_2}{2}\right)^2 \left(\frac{1}{ \lambda} 
  + \frac{1}{(1 - \lambda)}\right) \geq
\int_{\R} f \, \int_{\R} g .
\] 
\end{proof}
\vskip 5mm
\noindent
\subsection{Induction on the dimension}
\vskip 3mm
\noindent
\begin{proof}[Proof of Theorem \ref{thm:santalo_s_general}]
We prove the theorem by induction on the dimension.
The one-dimensional case was proved in 
Corollary \ref{cor:onedim_satalo_s_lambda}.
\par
\noindent
Assume now  that  the theorem is true in dimension $d-1$.
Let $b_{+} = \int_{H_{+}} xf(x) \di x$ and 
$b_{-}=\int_{H_{-}} xf(x) \di x,$ that is, 
$b_{\pm}$ is  a scalar multiple of the  barycenter of the restriction of the measure with density $f$  on $H_{\pm}$, respectively. Since $f$ is not  concentrated on $H$,   the point $b_{+}$ belongs to the interior of $H_{+}$, and similarly for $b_{-}.$ 
Hence the line passing through $b_{+}$ and $b_{-}$ intersects $H$ at one point, which we call $z$. We will show  that $z$ satisfies \eqref{eq:s_santalo_gen_onedim}, for all functions $g$ that satisfy the assumption. 
\par
\noindent
Clearly, replacing $f$ by $\shift{f}{-z}$ and $H$ by $H - z$, we can assume that $z=0$. Let $g$ be such that 
\begin{equation}
 \label{eq:reduality_lambda_santalo}
\forall x,y\in \R^d, \qquad f(x) g(y) \leq \truncf{1- \iprod{x}{y}}^s . 
\end{equation}
Let $e_1, \dots, e_d$ be an orthonormal basis of $\R^d$ such that $H = e_{d}^{\perp}$ and 
$\iprod{b_+}{e_d} > 0$. Let
$v = b_+ / \iprod{b_+}{ e_d}$ and let $A$ be the linear operator on $\R^d$ that maps $e_d$ to $v$ and $e_i$ to itself,  for $i=1\dots d-1$.  Let 
$B=\parenth{ A^{-1} }^t$. Define
\[ F_+ : H \to   \R_+  \hskip 3mm \text{by}   \hskip 3mm y_1 \mapsto \int_{\R_+} f( y_1 + t v) \, \di t  
\] 
and 
\[
G_+: H \to   \R_+  \hskip 3mm \text{by} \hskip 3mm y_2   \mapsto \int_{\R_+} g (B y_2 +t e_d) \, \di t . 
\] 
By Fubini's theorem, and since $A$ has determinant $1$, 
\begin{equation} \label{Fplus}
\int_H F_+  = \int_{H_+} f\circ A=\lambda \int_{\R^d} f.
\end{equation}
\par
\noindent
Also, letting $P$ be the projection with range $H$ and kernel $\R v$, one sees that the barycenter of $F_{+}$ is
\[
\frac{\int_H  x F_+ \di x }{\int_H   F_+ \di x } = 
\frac{\int_{H_+} P(Ax) f(Ax) \, \di x}{\lambda \int_{\R^d} f} =
P \Bigl(\frac{  \int_{H_+} x f(x) \, \di x }
{ \int_{H_{+}} f}\Bigr) = P ( b_+ ) , 
\] 
and this is $0$ by the definition of $P$. Since $\iprod{A x_1}{B x_2}
= \iprod{x_1}{x_2}$ for all $x_1,x_2 \in \R^d$, 
we have $ \iprod{ y_1 + t_1 v}{B y_2 + t_2 e_d } = 
\iprod{y_1}{y_2} + t_1 t_2$ for all $t_1 , t_2 \in \R$ and $y_1, y_2 \in H$. So \eqref{eq:reduality_lambda_santalo} implies
\begin{equation}
  \label{eq:santalo_dim_reduction}
  f ( y_1 + t_1 v ) g ( By_2 + t_2 e_d)  \leq \truncf{1 - t_1 t_2- \iprod{y_1}{y_2} }^s .
  \end{equation}  
Let  $t_1, t_2 > 0$  and assume   that $ \iprod{y_1}{y_2} < 1$.
Set $c = \sqrt{1 - \iprod{y_1}{y_2}}.$ Then, using \eqref{eq:santalo_dim_reduction} with $\tau_i = t_i/c,$ we get
\[
\frac{f(y_1 + {c v \cdot \tau_1 })g(B y_2 + {c e_d \cdot \tau_2 })}{c^{2s}} \leq \truncf{1 - \tau_1 \tau_2}^s
\]
By Lemma \ref{lem:s_santalo_onedim}, we get 
that
\[
\frac{1}{c^{2s}}\int_{[0, \infty)} f(y_1 + {c v \cdot \tau_1) } \di \tau_1  
\int_{[0, \infty)} g(B y_2 + {c e_d \cdot \tau_2 }) \di \tau_2  
\leq  \left(\frac{\sthing\kappa_2}{2}\right)^2.
\]
Returning to $t_1$ and $t_2$ in the integrals, one has
\[
F_{+} (y_1)\,  G_{+}(y_2) \leq \left(\frac{\sthing\kappa_2}{2}\right)^2 (1 - \iprod{y_1}{y_2})^{s+1} 
\]
for all $y_1, y_2$ in $H$ satisfying $\iprod{y_1}{y_2} < 1$.
\par
\noindent
If $\iprod{y_1}{y_2} \geq 1,$ inequality \eqref{eq:santalo_dim_reduction} implies that 
$F_{+} (y_1) G_{+}(y_2)= 0.$ Thus,
\begin{equation}
\label{eq:FG_marginal_santalo}
F_{+} (y_1)\,  G_{+}(y_2) \leq 
\parenth{\frac{\sthing\kappa_2}{2}}^2 \truncf{1 - \iprod{y_1}{y_2}}^{s+1}.
\end{equation}
The function $F_{+} $
is a function with finite integral on the $(d-1)$-dimensional space $H,$ and it is   $\frac{1}{1+s}$-concave  by the 
Borell--Brascamb--Lieb inequality \cite{brascamp1976extensions}. Thus, by induction assumption, there exists $v \in H$ such that
\[
 \int_{H} F_{+}  \int_{H} \slogleg[s+1]\!\parenth{{\shift{F_{+}}{v}}}
  \leq \parenth{\sthing[s+1]\kappa_{d}}^2.
\]
Such  a $v$ can be found in any hyperplane inside $H$ bisecting $F_{+}$.
Since the origin is the barycenter of $F_+$,  we see that the integral 
\[
\int_{H} \slogleg[s+1]\!\!\parenth{ \shift{\slogleg[s+1]F_{+}}{q}}
\] 
attains the minimum at the origin,   applying Theorem \ref{thm:Alexandrov_s-conc}  to $ \slogleg[s+1] F_{+}$.
Using again the induction assumption,  we can assume that $v = 0$ in the previous inequality. 
 Inequality \eqref{eq:FG_marginal_santalo} yields that $G_+$ is pointwise less or equal to 
$ \parenth{\frac{\sthing\kappa_2}{2}}^2 \slogleg[s+1]\!\parenth{F_{+}}.$
Hence, we get
\[
\int_{H} F_{+} \int_{H}{G_{+}} \leq 
{\parenth{\frac{\sthing\kappa_2}{2}}^2}  \, \int_{H} F_{+} \,   \int_{H} \slogleg[s+1]\!\parenth{F_{+}} \leq 
\parenth{\frac{\sthing\kappa_2}{2}}^2 
\parenth{\sthing[s+1]\kappa_{d}}^2. 
\]
However, by direct computation,
\[
 \volbs =\smeasure{\ball{d+1}} =   \pi^{d/2} 
  \frac{\gammaf{ s/2 +1}}{ \gammaf{ s/2 + d/2 + 1}},
\]
where $\gammaf{\cdot}$  is the Euler Gamma function, and thus 
$$
\parenth{\frac{\sthing\kappa_2}{2}}^2 
\parenth{\sthing[s+1]\kappa_{d}}^2= \parenth{\frac{\sthing\kappa_{d+1}}{2}}^2.
$$
Hence, also using (\ref{Fplus}), 
\[
\parenth{\frac{\sthing\kappa_{d+1}}{2}}^2 \geq 
\int_{H} F_{+}  \int_{H}{G_{+}}
= \lambda \int_{\R^d} f \int_{H_{+}} g(Bx) \di x = 
\lambda \int_{\R^d} f \int_{H_{+}} g.
\]
Similarly, 
\[
\parenth{\frac{\sthing\kappa_{d+1}}{2}}^2 \geq 
(1-\lambda) \int_{\R^d} f \int_{H_{-}} g.
\]
Therefore,
\begin{eqnarray*}
\int_{\R^d} f \int_{\R^d} g &=&
\int_{\R^d} f \left( \int_{H_{+}} g \ + \ \int_{H_{-}} g \right) \leq 
\parenth{\frac{\sthing\kappa_{d+1}}{2}}^2
\parenth{\frac{1}{\lambda} + \frac{1}{1 - \lambda}} \\
&= &
\parenth{\frac{\sthing\kappa_{d+1}}{2}}^2 \frac{1}{\lambda(1-\lambda)}.
\end{eqnarray*}
%Assume now that $\int_{\R^d} x f(x) \di x = 0.$
%For a unit vector $u,$ let $H(u)$ is a hyperplane with normal vector $u$  bisecting measure with density  $f(x) \di x.$ Denote the half-space with boundary $H(u)$
%containing $H(u) \pm u$ by $H_\pm (u),$ respectively;
%denote $b_\pm (u) = \int_{H_\pm (u)} x f(x).$ 
%By Borsuk--Ulam theorem,  there is $u$ such that $b_+(u) = - $
This completes the proof of Theorem \ref{thm:santalo_s_general}.
\end{proof}
%\begin{rem}
%This theorem can be proven without induction!
%We may use Lehec's version of Santalo inequality.
%I don't know whether this $\lambda$ thing is interesting...
%\end{rem}
\vskip 10mm
\section{Santal\'o \tpdfs-regions}
\label{sec:santalo_func}

In this section, we will define Santal\'o regions for functions and discuss possible approaches to define  Santal\'o functions. 
\par
\noindent
For a convex body $K$,  its \emph{Santal\'o region} $S\parenth{K, t}$  with parameter $t$ was  introduced  and studied in \cite{meyer1998santalo}. It is defined as
\begin{equation}\label{Santalo-body}
S\parenth{K, t} = \left\{ 
x \in K \st \vol{d} K  \cdot \vol{d}\shift{K}{x} \leq t \parenth{\vol{d} \ball{d}}^2
\right\}.
\end{equation}
Note that the Santal\'o region approximates the initial set as $t \to \infty$. More importantly, 
it was shown in \cite[Theorem 10]{meyer1998santalo} that  the {\em affine surface area}  of the initial convex body  can be computed as 
a limit as $t \to \infty$ of  the volume  difference of the convex body and its  Santal\'o regions. 
\par
Affine surface area was first introduced by Blaschke \cite{Blaschke1923} 
for 
dimensions $2$ and $3$ and for smooth enough convex bodies
as an integral over the boundary $\partial K$ of a power of the
Gauss curvature $\kappa_K$, 
$$
as(K) = \int_{\partial K} \kappa_K(x)^\frac{1}{n+1} d\mu(x).
$$
Integration is with respect to  the usual surface area measure $\mu$ on $\partial K$.
It was successively extend within the last decades. Aside from  the afore mentioned successful approach using the 
Santal\'o region, there are other successful approaches via the 
(convex) floating body resp. the illumination body  in   \cite{SchuettWerner1990, Werner1994} or in e.g., 
\cite{Lutwak1991, HanSlomkaWerner},  and they all coincide. Such extensions are desirable as 
the affine surface area is one of the most powerful tools in convex 
and differential geometry. It proved to be fundamental in 
the solution of the affine Plateau problem by Trudinger and Wang \cite{TrudingerWang2005,TrudingerWang2008}, in the theory 
of valuations where the affine and centro-affine surface areas have been  characterized by Ludwig and Reitzner 
\cite{Ludwig-Reitzner} and Haberl and Parapatits \cite{HaberlParapatitis} as unique valuations satisfying certain 
invariance properties.  Affine surface area appears naturally in the approximation of general 
convex bodies by polytopes, e.g.,  \cite{Boeroetzky2000, Reitzner, SchuettWerner2003}.
%These are only a few of the many applications (see, e.g.  also
%\cite{Andrews:1999,Andrews:1999a}), where the affine surface area reveals its 
%fruitfulness.
Furthermore,  there are  connections
to e.g.,  PDEs and ODEs and  concentration of volume (e.g., \cite{FGP2007, LutwakOliker}),  
information theory 
(e.g., \cite{artstein2012functional, CaglarWerner2014, 
PaourisWerner2012, Werner2012}) and in a spherical and hyperbolic setting \cite{BesauWerner2015a, BesauWerner2018}. 
\par
\noindent
%We find it is not exciting to approximate the support of a $1/s$-function. 
It would be extremely interesting to develop  an analogue of  affine surface area in the functional setting.  
There are several approaches already how to define it for log-concave functions \cite{artstein2012functional, CaglarWerner2014, caglar2016functional, LiSchuettWerner},
but those do not all coincide. 
We think that an approach via  a Santal\'o function will help not only to clarify this point, but also 
can reveal new properties and inequalities related to log-concave functions and their integrals. 
Towards this goal, we  define  Santal\'o functions  for $1/s$-concave functions  in the next subsection.
%It is our goal  to only point out here several  possible approaches. Therefore,  we do not prove technical results stated in this section.
\vskip 2mm
\noindent
%The \emph{Santal\'o region} $S\parenth{K, t}$ of a convex body $K$ in $\R^d$ is defined as follows
%\[
%S\parenth{K, t} = \left\{ 
%x \in K \st \vol{d} K  \cdot \vol{d}\shift{K}{x} \leq t \parenth{\vol{d} \ball{d}}^2
%\right\}.
%\]
%The Santal\'o region of a convex body in $\mathbb{R}^d$ was introduced and studied  in \cite{meyer1998santalo}. 
First we  define, following  (\ref{Santalo-body}),  the Santal\'o $s$-region for a fixed positive $s$ by
 $\santalosreg{f}{s}{t}$ of a non-negative Borel function $f$ on $\R^d$   by
 \[
 \santalosreg{f}{s}{t} = \left\{ 
 x \in \co \supp f \st 
 \int_{\R^d} f \cdot \int_{\R^d} \slogleg \parenth{\shift{f}{x}} \leq 
 t \parenth{\volbs}^2 
 \right\},
 \]
where $\text{co}\, A$ denotes the convex hull of a set $A$.
In the limit case $ s =\infty$ 
 we define the Santal\'o $\infty$-region by
 \[
 \santalosreg{f}{\infty}{t} = \left\{ 
 x \in \co \supp f \st 
 \int_{\R^d} f \cdot \int_{\R^d} \slogleg[\infty] \parenth{\shift{f}{x}} \leq 
 t  \cdot (2 \pi)^d
 \right\}.
 \]

 \vskip 2mm
 \noindent
 We summarize the properties of the Santal\'o regions in the next lemmas. They 
  follow from Lemmas \ref{lem:dual_s_lifting} and \ref{lem:int_s-polar_via_suppfunc} and Theorems \ref{thm:Alexandrov_log-conc}-\ref{thm:santalo_s_general}.
\vskip 2mm
\noindent
\begin{lem}
Let $s$ be a fixed positive real number.
Let $f$  be a non-negative function on $\R^d$ such that 
$\slogleg \slogleg f$ has positive integral. Then  $\santalosreg{f}{s}{t}$  
\begin{enumerate}
\vskip 2mm
\item  is non-empty if $t \geq 1,$ and has non-empty interior if $t > 1$,
\vskip 2mm
\item  is a convex set if it is non-empty,
\vskip 2mm
\item  is strictly convex if it has non-empty interior,
\vskip 2mm
\item  has $C^{\infty}$-smooth boundary if it has non-empty interior.
\end{enumerate}
\end{lem}
\vskip 2mm
\noindent
\begin{lem}
Let $f \colon \R^d \to [0, \infty)$ be a proper log-concave function.
If  $\santalosreg{f}{\infty}{t}$ is non-empty, then 
\[
 \santalosreg{f_s}{s}{t}  \to  \santalosreg{f}{\infty}{t} 
\] 
in the Hausdorff metric as $s \to \infty.$
\end{lem}

\vskip 5mm

\subsection{Marginals of convex sets}

\par
Let $s$ be the reciprocal of a positive integer. 
Then, as we already discussed,  any $1/s$-concave function on $\R^d$ is the marginal of a convex set in $\R^{d+s}$. 
In particular, one can lift  a $1/s$-concave function $f \colon \R^d \to [0, \infty)$ into $\R^{d+ s}$ as follows \cite{artstein2004santalo}, 
\begin{equation}\label{def.Ksf}
{K}_s(f) = \left\{
(x, y) \in \R^{d} \times \R^{s} \st x \in \closure \supp f,\, \enorm{y} \leq \parenth{\frac{f(x)}{\vol{s} \ball{s}}}^{\!1/s}
\right\}.
\end{equation}
Clearly,
\[
\vol{d +s} {K}_s(f) =  \int_{\supp f} 
\vol{s} \ball{s} \parenth{ \frac{f(x)}{\vol{s} \ball{s}}} \di x =
\int_{\R^d} f. 
\]
From \cite[Lemma 3.1]{artstein2004santalo} follows that for any $z$ in the interior of the support of $f,$ one has
\[
\vol{d+s} \shift{K_s(f)}{z} = 
\parenth{\vol{s} \ball{s}}^2 
\int_{\R^d}  \slogleg\! \parenth{\shift{f}{z}},
\]
and therefore, 
\[
 \int_{\R^d} f  \,  \int_{\R^d} \slogleg \! \parenth{\shift{f}{z}}  =
  \frac{\vol{d+s} K_s(f) \,  \vol{d+s} \shift{K_s(f)}{z}}
{\parenth{\vol{s} \ball{s}}^2}.
\]
On the other hand, $\vol{d+s} \ball{d+s}= \volbs \vol{s} \ball{s}$.
\vskip 2mm
\noindent
We then  define the  {\em Santal\'o $m$-function}  $S_{m}(f,s,t)$ of a $1/s$-concave function $f$ with integer $s$ to be such that
 \begin{equation}\label{Santalofunction-sconcave}
K_s \! \parenth{S_m(f,s,t)}= 
 S \! \parenth{K_s(f),  t}.
\end{equation}
\par
\noindent
It follows that the Santal\'o m-function $S_{m}(f,s,t)$ of a $1/s$-concave function $f$ is $1/s$-concave. 
Moreover,  the following identity holds.
\vskip 2mm
\noindent
\begin{prp} Let $s \in \N,$ and let $f \colon  \R^d \to [0, \infty)$ 
be a  $1/s$-concave function of finite integral. Then
\begin{eqnarray*}
&&\hskip -10mm \lim_{t \to \infty}  
\frac{  \int_{\R^d} f  -  \int_{\R^d} S_m(f,s,t) }
{ t^{-\frac{2}{ d+ {s} +1 }} }= \\
&& \hskip 20mm \frac{s}{2} \,
\parenth{  \frac{\vol{s} \ball{s}}{\vol{d+s} \ball{d+s}} \int_{\R^d} f}^\frac{2}{d+s+1}  \, 
 \int_{\R^d}  \left| \det \parenth{\text{Hess} \left(f^\frac{1}{s}}
 \right) \right|^\frac{1}{d+s+1}
f^{\frac{(s-1)(d+s)}{s(d+s+1)} }, 
\end{eqnarray*}
where $\text{Hess} \left(f^\frac{1}{s}\right)$ is the Hessian of $ f^\frac{1}{s}.$
\end{prp}
\vskip 2mm
\noindent
\begin{proof}
The proof follows immediately from Theorem 10 of  \cite{meyer1998santalo} and  Proposition 6 of \cite{artstein2012functional}.
\end{proof}
\vskip 3mm 
\noindent
%\subsection{Convex functional}
It remains to find a reasonable definition of a  Santal\'o function for a log-concave function. The natural first approach 
$\lim_{s \to \infty} S_m(f_s,s,t)$ does not lead to anything meaningful.
Lemma \ref{lem:int_s-polar_via_suppfunc}  and Theorem \ref{thm:Alexandrov_s-conc} provide another possible approach.
%give us another clue of how to define  Santal\'o functions.
%Our idea is as follows. 
\par
\noindent
For  $s>0$,  we  define the set 
\begin{eqnarray*}
&&\sthing S_p(f, s, t) =\\
&& \left\{
y \in \slift{f} \st 
\frac{s}{2(d+s)}\int_{S^{d}} 
\frac{\abs{\iprod{e_{d+1}}{ u}}^{s-1}}{
\left(h_{\shift{K}{z}}(u)\right)^{d+s}} \di \spherem{u} \cdot {\int_{\R^d} f  }   \leq t \parenth{ \volbs }^2
\right\}.
\end{eqnarray*}
Lemma \ref{lem:int_s-polar_via_suppfunc} shows  that $\sthing S_p(f, s, t)$ is a convex $d$-symmetric set, and hence, it is the 
$s$-lifting of a $1/s$-concave function which we denote $S_p(f, s, t)$.
The function $S_p(f, s, t)$  seems a  good candidate for a  Santal\'o function and we investigate this in a forthcoming paper.

%However, it's extremely difficult to compute them.

%\bibliographystyle{amsalpha}
%\bibliography{../../work_current/uvolit}

%Institute of Science and Technology Austria (IST Austria), 
%Klosterneuburg, 3400, Austria; Laboratory of Combinatorial and Geometrical Structures, Moscow Institute of Physics and Technology, Moscow, 141701, Russia}
%\email{grimivanov@gmail.com}

\vskip 20mm
\noindent
Grigory Ivanov \\
{\small Institute of Science and Technology Austria }\\
{\small Klosterneuburg, 3400, Austria} \\ 
%{\small and}
%{\small Moscow, Russian Federation}   \\
{\small \tt grimivanov@gmail.com}

\vskip 10mm
\noindent
Elisabeth M. Werner\\
{\small Department of Mathematics \hskip 42 mmUniversit\'{e} de Lille 1}\\
{\small Case Western Reserve University \hskip 34mm UFR de Math\'{e}matique }\\
{\small Cleveland, Ohio 44106, U. S. A. \hskip 36mm  59655 Villeneuve d'Ascq, France}\\
{\small \tt elisabeth.werner@case.edu}\\

\end{document}